\documentclass[9pt,journal,final]{IEEEtran}
\ifCLASSINFOpdf
\else
\fi

\usepackage{amsmath}
\usepackage{amsfonts}
\usepackage{amssymb}
\usepackage{verbatim} 
\usepackage{graphics,psfrag,epsfig}
\usepackage{epstopdf} 
\usepackage[noadjust]{cite}
\usepackage{empheq}
\usepackage{cases}

\usepackage[linesnumbered,ruled,vlined,algo2e]{algorithm2e}

\usepackage{color}
\definecolor{red}{rgb}{1,0.2,0.2}
\definecolor{green}{rgb}{0.2,1,0.5}
\definecolor{blue}{rgb}{0,0,1}
\definecolor{lightblue}{rgb}{0.3,0.5,1}

\newcommand{\diag}{\mathrm{diag}}
\newcommand{\card}{\mathrm{card}}

\newcommand{\tr}{\mathrm{tr}}
\newcommand{\ds}{\mathrm{d}}

\newcommand{\0}{\mathbf{0}}
\newcommand{\1}{\mathbf{1}}
\newcommand{\Mn}{\mathbb{R}^{N \times N}}

\newcommand{\RN}{\mathbb{R}^{N}}

\newcommand{\e}{\mathbf{e}}

\newcommand{\T}{^{\mathsf{T}}}
\newcommand{\nT}{^{-\mathsf{T}}} 
\newcommand{\uu}{\mathbf{u}}
\newcommand{\vv}{\mathbf{v}}

\newcommand{\x}{{\mathbf{x}}}
\newcommand{\y}{\mathbf{y}}

\newcommand{\bb}{\mathbf{b}}
\newcommand{\cc}{\mathbf{c}}

\newcommand{\sbold}{\mathbf{s}}
\newcommand{\sB}{\mathbf{s}}
\newcommand{\ub}{\mathbf{u}}
\newcommand{\vb}{\mathbf{v}}

\newcommand{\xx}{{\mathbf{x}}}

\newcommand{\nnb}{\nonumber}

\usepackage{amsthm}
\usepackage{hyperref}
\newtheorem{thm}{\bf Theorem}
\newtheorem{lem}{\bf Lemma}

\newtheorem{asmp}{\bf Assumption}
\newtheorem{cor}[thm]{\bf Corollary}

\theoremstyle{remark}
\newtheorem{rem}{\bf Remark}

\makeatletter
\def\BState{\State\hskip-\ALG@thistlm}
\makeatother

\begin{document}
	\title{Optimizing Leader Influence in Networks through Selection of Direct Followers}
	
	\author{Van Sy Mai and Eyad H. Abed
		\thanks{Van Sy Mai and Eyad H. Abed are with the Department of Electrical and Computer Engineering and the Institute for Systems Research, 
			University of Maryland, College Park, MD 20742, USA. 
			{\tt\small \{vsmai,abed\}@umd.edu}}%
	}
	

	\hyphenation{op-tical net-works semi-conduc-tor}

	\maketitle
	
	\begin{abstract}
		The paper considers the problem of a leader that seeks to optimally influence the opinions of agents in a directed network through connecting with a limited number of the agents (``direct followers''), possibly in the presence of a fixed competing leader. The settings involving a single leader and two competing leaders are unified into a general combinatoric optimization problem, for which two heuristic approaches are developed. The first approach is based on a convex relaxation scheme, possibly in combination with the $\ell_1$-norm 
		regularization technique, and the second is based on a greedy selection strategy. The main technical novelties of this work are in the establishment of supermodularity of the objective function and convexity of its continuous relaxation. The greedy approach is guaranteed to have a lower bound on the approximation ratio sharper than $(1-1/e)$, while the convex approach can benefit from efficient (customized) numerical solvers to have practically comparable solutions possibly with faster computation times. The two approaches can be combined to provide improved results. In numerical examples, the approximation ratio can be made to reach $90\%$ or higher depending on the number of direct followers.

	\end{abstract}
	

	%
	\IEEEpeerreviewmaketitle

	\section{Introduction}
	
	This paper revisits problems related to a leader seeking to influence the opinions of agents in a strongly connected network through connecting with a limited number of the agents, referred to here as ``direct followers.'' The leader has a constant opinion and aims to achieve maximum influence on the agents' opinions, which evolve through iterative updating as a  weighted average of the opinions of their neighbors. Our main contributions include:
	\begin{itemize}
		\item \emph{Defining alternative and more practical influence maximization problems in which the network is directed and weighted}, and direct followers' opinions are not anchored but follow dynamics. We consider two situations: (i) there is a single leader, and (ii) the preferred leader operates in the presence of a competing leader. We formulate influence maximization problems for these scenarios, each involving a tailored measure of performance of the preferred leader in influencing the network agents. In case (i), we seek to minimize the cumulative distance of agent opinions from that of the preferred leader, while in case (ii) we focus on the steady state distance.
		
		\item 
		\emph{Unifying the two problems into a common mathematical framework and providing practical solution methods}. 
		We embed these two influence optimization problems in a single, general combinatoric optimization problem. We develop and analyze \emph{convexity heuristics} and \emph{greedy algorithms} for the unified problem; these can be treated effectively by available numerical algorithms.
	\end{itemize}

	This paper is related to a large body of literature on problems of leader selection and stubborn agent placement (see, e.g., \cite{ Patterson10CDC, Lin11, Fardad13, Ghaderi13, Gionis13, Yildiz13ACM, Clark14AC, Vassio14, Borkar15} and references therein) 
	but departs from this literature in several key aspects. First, we only ask that the underlying network be directed and strongly connected. Second, we allow  selected direct follower nodes to follow inter-agent dynamics like other agents, rather than instantaneously anchoring their opinions to that of the leader.  
	Third, we allow the agents in the network to have different initial opinions (which are taken into account explicitly in the case of having one leader), and the agents can be assigned differing weights by the leader. 
	Finally, and more importantly, although continuous relaxation and greedy heuristics have been employed in influence maximization problems,  our theoretical results on convexity and supermodularity are considerably stronger than existing results, without assuming any symmetry or resorting to the random walk theory. This not only provides a deeper understanding of diffusive processes but also can be used for a broader range of applications. 
	We compare our results with related work as appropriate.


	The remainder of the paper proceeds as follows. 
	In Section \ref{secPropblemFormulation}, we introduce our network models and associated optimization problems of interest. 
	Related work is also reviewed. 
	Our main results are given in Sections \ref{secCVX} and \ref{secSupermodularity}. In Section \ref{secCVX}, we establish the convexity of the relaxed and approximate problems and discuss associated numerical issues in applying convex solvers to these problems. In Section \ref{secSupermodularity}, we prove the supermodularity property of the original objective functions and present two greedy algorithms that admit provable approximation ratios. 
	Finally, a few simulation results are reported in Section \ref{secSimulations}. 
	
	\textit{Notation and terminology: } 
	The real part of a complex number $x$ is denoted $\Re(x)$. 
	Vectors are denoted by bold lower case letters, e.g., 
	$\x = [x_1, x_2, \ldots, x_n]\T\in \mathbb{R}^n$ 
	and $\1 := [1,1,\ldots,1]\T$. 
	For any $\x\in \mathbb{R}^n$, $| \x | \in \mathbb{R}^n$ is such that its $i$th element is $|x_{i}|$, $\card(\x)$ denotes the number of nonzero elements of $\x$,  and $\diag(\x) \in \mathbb{R}^{n\times n}$ is the diagonal matrix with $[\diag(\xx)]_{ii} = x_i, i=1,\ldots, n$. 
	For a matrix $A$, $\|A\|_p$ denotes its $p$-norm, 
	$\rho(A)$ its spectral radius, $\sigma(A)$ its spectrum, 
	and $A_{(i)}$ and $A^{(j)}$ the $i$-th column and $j$-th row of $A$, respectively. 
	Any of $[A]_{ij}$, $A_{ij}$ and $a_{ij}$ can be used to indicate the $ij$-th element of $A$. 
	The identity and zero matrices are denoted by $I$ and $\0$, respectively (dimensions will be clear from the context). 
	%
	%
	We write $A\!\ge\! B$ when $A-B$ is a nonnegative matrix.
	A square matrix $A$ is {\it substochastic} if $A\!\ge\! \0$ and $A\1\!\le\!\1$. 
	A square matrix $A$ is an {\it M-matrix} if $a_{ij} \!\le\! 0, \forall i\neq j$ and $A \!=\! sI-B$ for some $s>0$ and $B\ge \0$ satisfying $\rho(B) \le s$. 
	
	Sets are denoted by calligraphic upper case letters. 
	For a set $\mathcal{A}$, $|\mathcal{A}|$ denotes its cardinality. 
	A {\it directed graph} $\mathcal{G} \!=\! (\mathcal{V}, \mathcal{E})$ consists of a finite set of nodes (also called agents in this work) $\mathcal{V} = \{1,2,\ldots,N\}$ and a set $\mathcal{E} \subseteq \mathcal{V}\times \mathcal{V}$ of edges, where $(i,j) \in \mathcal{E}$ is an ordered pair indicating that agent $i$ can obtain information directly from agent $j$. A {\it directed path} is a sequence of edges in the form $(i_1, i_2), (i_2, i_3),\ldots, (i_{k-1}, i_k)$. 
	Node $j$ is said to be {\it reachable from node} $i$ if there is a directed path from $i$ to $j$. Each node is reachable from itself. The graph $\mathcal{G}$ is {\it strongly connected} if each node is reachable from any other node. 
	
	Let $\mathcal{V}$ be a finite set. A function $f:2^\mathcal{V}\to \mathbb{R}$ is {\it supermodular} if $f(\mathcal{S}) - f(\mathcal{S} \cup \{v\}) \geq f(\mathcal{T}) - f(\mathcal{T} \cup \{v\})$ for any $\mathcal{S} \subseteq \mathcal{T} \subseteq \mathcal{V}$ and $\forall v\in \mathcal{V}{\setminus} \mathcal{T}$; we say that $f$ is {\it submodular} if $-f$ is supermodular. 
	
	A differentiable function $f$ is called {\it strongly convex with parameter} $\mu>0$ {\it on a convex set} $\Omega$ if for any $ \x,\y\in \Omega$, $f(\y) \ge f(\x) + \nabla f(\x)\T(\y-\x) + \frac{\mu}{2}\|\y-\x\|^2$, 
	where $\nabla f$ denotes the gradient. If $\mu \ge 0$ in this relation, the function $f$ is {\it convex}. 
	
	\vspace{-2mm}
	\section{Problem Formulation and Related Work}\label{secPropblemFormulation}
	
	\subsection{Opinion Dynamics Model and Assumptions}\label{subsecDeGrootModel}
	Consider a network of $N$ agents, characterized by a graph $\mathcal{G} = (\mathcal{V}, \mathcal{E})$, in the presence of two leaders with different opinions $T$ and $Q$; we refer to $T$ and $Q$ as competing leaders. 
	Let $\mathcal{K}, \mathcal{L} \subseteq \mathcal{V}$ denote the sets of direct followers of $T$ and $Q$, respectively. 
	Let $x_i(t) \in [0,1]$ denote the opinion of node $i$ at time $t \ge 0$. 
	Each node $i$ has two potential trust levels $\alpha_i, \beta_i\in [0,\infty]$ (at least one of them is finite) and updates its opinion according to
	\begin{align}
	{x}_i(t+1) = \displaystyle\frac{[\sbold_{\mathcal{K}}]_i\alpha_i T + [\sbold_{\mathcal{L}}]_i\beta_i Q+ \sum_{j\in \mathcal{N}_i} w_{ij}x_j(t)}{[\sbold_{\mathcal{K}}]_i\alpha_i+[\sbold_{\mathcal{L}}]_i\beta_i+ \sum_{j\in \mathcal{N}_i} w_{ij}},
	\label{eqModel2Leader}
	\end{align}
	where $w_{ij} >0$ represents the weight agent $i$ puts on agent $j$, $\mathcal{N}_i$ the set of agent $i$'s neighbors, and $\sbold_{\mathcal{K}}$ and $\sbold_{\mathcal{L}}$ denote the selection vectors of $T$ and $Q$, respectively, i.e., $[\sbold_{\mathcal{K}}]_i=1$ if $i\in \mathcal{K}$ and $0$ otherwise. 
	In our context, $\boldsymbol{\alpha}$ and $\boldsymbol{\beta}$ are associated with the agents and are  assumed to be fixed over time. This model is a special case of the Friedkin model \cite{Friedkin15CS} and reduces to the DeGroot model \cite{DeGroot74} (see also \cite{Jadbabaie03,Olfati07,Ren07,Arcak07}) when $\sbold_{\mathcal{K}} = \sbold_{\mathcal{L}} = 0$, i.e., both $T$ and $Q$ are inactive.  	In this paper, our interest is in maximizing the influence of leader $T$ on the opinions of the network agents, through selection of up to $K$ direct followers. 
	We make the following blanket assumptions:
	
	\begin{asmp} \label{asmp_fixedGraph} 
		The graph $\mathcal{G}$ is strongly connected and the weight matrix $W = [w_{ij}]$ is such that for $i\neq j$,  
		$w_{ij}>0$ if $(i,j)\in \mathcal{E}$ and $w_{ij}=0$ otherwise. Moreover, $w_{ii}>0$ for some $i\in \mathcal{V}$. 
	\end{asmp}

	\begin{asmp} \label{asmp_alpha} 
		$\mathcal{K} \cap \mathcal{V}_{\alpha} \neq \varnothing$ where $\mathcal{V}_{\alpha} := \{i\in \mathcal{V}~|~  \alpha_i > 0 \}$. 
	\end{asmp}
	
	Assumption \ref{asmp_alpha} involves no loss of generality; it simply requires that at least one direct follower has a nonzero trust level in $T$.
	
	We consider two cases: either $T$ is the only active leader ($\sB_{\mathcal{L}} \!=\! \0$), or the second leader $Q$ is also present and active  ($\sB_{\mathcal{L}} \!\neq\!\0$). In the latter case,  leader $Q$ has chosen its direct followers, and this choice is known. In this case, since both $T$ and $Q$ are active, the agents' opinions do not reach consensus, while they do reach consensus in the case of a single active leader $T$. Because of the differing network behaviors, we will use different measures of influence for leader $T$ for the two cases. Although we use two different performance measures for the two cases considered, mathematically the two problems can be cast as instances of a single problem. This unified mathematical problem is best introduced after deriving the individual optimization problems for the two cases and then comparing the formulations.
	
	\vspace{-2mm}
	\subsection{Influence Optimization for the Single Leader Case}\label{subsecWithOneLeader}
	Letting $\sB_{\mathcal{L}} = \0$, the model reduces to
	\begin{align}
	{x}_i(t+1) = 
	\frac{[\sbold_{\mathcal{K}}]_i\alpha_i T+ \sum_{j\in \mathcal{N}_i}w_{ij}x_j(t)}{[\sbold_{\mathcal{K}}]_i\alpha_i+ \sum_{j\in \mathcal{N}_i} w_{ij}}.
	\label{eqModel1Leader}
	\end{align} 
	The following well known result (see, e.g., \cite{Jadbabaie03,Ghaderi13,ParsegovAC17}) asserts that all network agents will adopt the leader's opinion asymptotically, regardless of their initial opinions.
	
	\begin{thm} \label{thm_consensus} 
		Under Assumptions \ref{asmp_fixedGraph} 
		and \ref{asmp_alpha}, 
		$\lim_{t\to \infty} x_i(t) = T, \forall i\in \mathcal{V}, \x(0) \in [0,1]^N$. Moreover, the rate of convergence is linear. 
	\end{thm} 
	
	
	Although having no role in the final consensus value, $\x(0)$ and  $\mathcal{K}$ clearly affect the manner in which the agents approach this agreement. 
	To examine this dynamic behavior, we consider the transient error $\boldsymbol\xi(t) := \boldsymbol\x(t)-T\mathbf{1}$, which by \eqref{eqModel1Leader} satisfies
	\begin{align}
	\boldsymbol\xi(t+1) &= A \boldsymbol\xi(t),\quad A := (\diag(W\1 + \boldsymbol\alpha_{\mathcal{K}}))^{-1}W, \label{eqError1Leader}
	\end{align} 
	where $\boldsymbol\alpha_{\mathcal{K}} {=} \sB_{\mathcal{K}} {\diamond} \boldsymbol\alpha$ and $\diamond$ denotes the element-wise product (also known as the Hadamard product). 
	Asymptotically approaching consensus is equivalent to global exponential stability of the origin for \eqref{eqError1Leader}, since \eqref{eqError1Leader} is linear and time-invariant. 
	In fact, we have the following spectral properties of $A$ and a related matrix, the (weighted) Laplacian matrix $L := \diag(W\1) - W$. 
	\begin{lem} \label{lemStability}
		\emph{(Spectrum)}	
		If Assumptions \ref{asmp_fixedGraph} and \ref{asmp_alpha} hold, then 
		\begin{itemize}
			\item[(i)] $\rho\big( (\diag(W\1 + \boldsymbol\alpha_{\mathcal{K}}))^{-1}W \big) < 1$, and 
			
			\item[(ii)] 
			$\forall \lambda \in \sigma\big(L+\diag(\boldsymbol\alpha_{\mathcal{K}}) \big), \Re(\lambda) > 0$.
		\end{itemize} 
	\end{lem}
	\begin{proof}
		Clearly, $A \!=\! (\diag(W\1 \!+\! \boldsymbol\alpha_{\mathcal{K}}))^{-1}W$ is irreducible (i.e., graph associated with $A$ is strongly connected) and substochastic with at least one row sum less than 1. It follows from \cite[Thm. 1.1, p. 24]{Minc88} that $\rho(A) {<}1$. Part (ii) follows from  the Gershgorin Circle Theorem \cite[p. 344]{Horn85} and \cite[Cor.~6.2.9, p. 356]{Horn85}.%
	\end{proof}
	

	We are interested in the total convergence error defined as $$\textstyle J^{total}_\mathcal{K} =  \sum_{i\in \mathcal{V}}b_i\sum_{t=1}^{\infty}|\xi_i(t)|,$$ 
	where $\bb {=} [b_1,..., b_N]\T \!\ge\! \0$ is a weight vector satisfying $\1\T\bb \!=\! 1$. We assume that $\bb$ is fixed, chosen a priori by the leader, to reflect the leader's relative preferences for the various available agents. 
	When $\bb{=}\1/N$ and $\boldsymbol{\xi}(0) {=} -\1$, 
	$J^{total}_\mathcal{K}$ is similar to error measures considered in, e.g.,  \cite{Fardad13,Clark14AC,Borkar15}. 
	%
	However, since computing $J^{total}_{\mathcal{K}}$ for any $\boldsymbol{\xi}(0) \!=\! \boldsymbol{\xi}_0$ is nontrivial, we employ a tight upper bound $J_{\mathcal{K}}^{(1)}$ obtained as follows: 
	\begin{align} 
	J^{total}_{\mathcal{K}}  
	&\!=\! \textstyle \sum_{t\ge 1} \bb\T |\boldsymbol\xi (t)| \stackrel{\eqref{eqError1Leader}}{=} \textstyle \bb\T\! \sum_{t\ge 0} | A^t\boldsymbol\xi (1)| \le \textstyle \bb\T\! \sum_{t\ge 0} A^t |\boldsymbol\xi (1)| \nnb\\
	&\!\!\!\!\!\stackrel{(\text{Lem.~\ref{lemStability}})}{=}  
	\bb\T(I - A)^{-1} |\boldsymbol\xi(1)|\nnb\\ 
	&=  \bb\T(\diag(W\1 \!+\! \boldsymbol\alpha_{\mathcal{K}}) \!-\! W)^{-1} \diag(W\1 \!+\! \boldsymbol\alpha_{\mathcal{K}})|\boldsymbol\xi(1)| \nnb\\
	&\le \bb\T(L+ \diag(\boldsymbol\alpha_{\mathcal{K}}))^{-1} |W\boldsymbol\xi_0| ~=: J_{\mathcal{K}}^{(1)}.  \label{eqJK1d}
	\end{align}
	Here the last inequality holds since the inverse $(L+ \diag(\boldsymbol\alpha_{\mathcal{K}}))^{-1}$ exists (cf. Lemma \ref{lemStability}(ii)) and  $|\boldsymbol\xi(1)| \le (\diag(W\1 + \boldsymbol\alpha_{\mathcal{K}}))^{-1}|W\boldsymbol\xi_0|.$ 
	It can be verified that $J_{\mathcal{K}}^{(1)} = J^{total}$ if either $\boldsymbol\xi_{0} \ge \0$ or $\boldsymbol\xi_{0} \le \0$, 
	i.e., if $T \ge \max_i x_i(0)$ or $T \le \min_i x_i(0)$ 
	(this is the case when viewing $T$ as a new idea or product that is being promoted). 
	Thus $J_{\mathcal{K}}^{(1)}$ can be regarded as a cost associated with leader $T$ during the transient process (i.e., prior to the whole network adopting opinion $T$).   
	We use $J_{\mathcal{K}}^{(1)}$ to measure the influence of $T$ and consider:
	\begin{equation} \label{eqMinJK1}
	\quad\textrm{(P1)} \quad
	\displaystyle{ \min_{\mathcal{K} \subset \mathcal{V}, |\mathcal{K}| \le K}} \quad J_{\mathcal{K}}^{(1)} = \bb\T(L+ \diag(\boldsymbol\alpha_{\mathcal{K}}))^{-1} |W\boldsymbol\xi_0|.
	\end{equation}
	

	\vspace{-5mm}
	\subsection{Influence Optimization in the Presence of a Competing Leader}\label{subsecWithTwoLeaders}
	Next, consider the case  $\sB_{\mathcal{L}} \neq \0$, i.e., leader $T$ competes against leader $Q$. Without loss of generality, suppose $\boldsymbol\beta\neq \0$ is fixed and $\sbold_{\mathcal{L}} = \1$ (i.e., the direct followers of $Q$ are fixed and known). 
	In this case, it is well known that the agents' opinions do not reach consensus but converge to a limiting opinion vector ${\x}(\infty)$ satisfying  
	${\x}(\infty) = (\diag(W\1 + \boldsymbol\alpha_{\mathcal{K}}+ \boldsymbol\beta))^{-1}(\boldsymbol\beta Q + \boldsymbol{\alpha}_{\mathcal{K}}T+ W \x(\infty))$. 
	Thus, 
	\begin{equation}
	\x({\infty})  = (L_\beta + \diag(\boldsymbol\alpha_{\mathcal{K}}) )^{-1}(\boldsymbol\beta Q + \boldsymbol{\alpha}_{\mathcal{K}}T),  \nnb
	\end{equation}
	where $L_{\beta} {:=} L {+} \diag(\boldsymbol{\beta})$, which 
	is invertible under the strong connectivity assumption and the condition that $\boldsymbol{\alpha}_{\mathcal{K}} \neq \0$ and $\boldsymbol{\beta} \neq \0$ (cf. Lemma \ref{lemStability}-ii). 
	We are interested in the steady state error vector $\boldsymbol\xi(\infty) := \boldsymbol\x(\infty)-T\1$. Since $(L_\beta + \diag(\boldsymbol\alpha_{\mathcal{K}}) )^{-1}(\boldsymbol{\beta}+\boldsymbol\alpha_{\mathcal{K}}) = \1$, it can be verified that 
	$\boldsymbol\xi({\infty})  = (L_\beta + \diag(\boldsymbol\alpha_{\mathcal{K}}) )^{-1}\boldsymbol\beta (Q-T)$. 
	To quantify the long term effect of $T$ in the presence of $Q$, we define 
	$$ J_{\mathcal{K}}^{(2)} := \bb\T |\boldsymbol\xi({\infty})|,$$ 
	where $\bb \!\ge\! \0$ is a preference vector. 
	Since $(L_\beta + \diag(\boldsymbol\alpha_{\mathcal{K}}) )$ is a nonsingular M-matrix, we have $(L_\beta + \diag(\boldsymbol\alpha_{\mathcal{K}}) )^{-1} \ge \0$  (cf. Lemma~\ref{thm_inverseMmatrix} in Appendix~\ref{secMatrixResults}). Thus
	$ J_{\mathcal{K}}^{(2)} \!=\! \bb\T (L_\beta \!+\! \diag(\boldsymbol\alpha_{\mathcal{K}}) )^{-1}\boldsymbol{\beta}|Q{-}T|.$ 
	Hence, letting ${T=0}$ and ${Q=1}$, we are interested in the problem:
	%
	\begin{align} 	\label{eqMinJK2}
	\quad\textrm{(P2)} \quad 
	\displaystyle{\min_{\mathcal{K}\subset \mathcal{V}, |\mathcal{K}| \le K}} \quad  J_{\mathcal{K}}^{(2)} = \bb\T (L_\beta + \diag(\boldsymbol\alpha_{\mathcal{K}} ))^{-1}\boldsymbol{\beta}.
	\end{align} 
	
	
	In the limiting case when $\alpha_i, \beta_i$ are either $0$ or $\infty$, (P2) reduces to the previously studied optimal stubborn placement or leader selection problems in the literature, recalled below. First, we give a general problem formulation that covers both (P1) and (P2).

	\begin{rem}\emph{(A unified problem formulation)} 
		Except for some minor differences, (P1) and (P2) are almost the same. We thus seek methods that apply to both. To this end, we embed them in the more general problem 
		\begin{equation} \label{eqMinJK}
		\quad\textrm{(P)} \qquad 
		\displaystyle{\min_{\mathcal{K}\subset \mathcal{V},|\mathcal{K}| \le K}} \quad  J_{\mathcal{K}} = \bb\T (L_\beta + \diag(\boldsymbol\alpha_{\mathcal{K}} ))^{-1}\cc,
		\vspace{-1mm}
		\end{equation} 
		where $\bb,\cc,\boldsymbol\beta \in \RN_+$. 
		The optimal value will be denoted by $J^*$. 
		As will be shown later (in Theorems~\ref{thmConvexity_G} and \ref{thmSupermodularJK} below), $J_{\mathcal{K}}$ is nonincreasing in $\mathcal{K}$. Thus, it follows that any solution to (P) is also compliant with Assumption~\ref{asmp_alpha} unless (P) is trivial (i.e., $\mathcal{K}=\varnothing$ is also optimal). 
	\end{rem}


	\vspace{-3mm}
	\subsection{Comparison to Previous Work}

	\subsubsection{Single leader case}
	The following model is widely used in the literature (see, e.g., \cite{Olfati07,Fardad13,Borkar15,Mai16a}):
	\begin{align}
	x_i(t\!+\!1) \!=\! \begin{cases} 
	\tilde\alpha_i T \!+\! \left(1 \!-\!\tilde\alpha_i \right) \sum_{j\in \mathcal{V}}{\tilde w_{ij} x_j(t)}, ~~i\in \mathcal{K} \\
	\sum_{j\in \mathcal{V}}{\tilde{w}_{ij}x_j(t)},  \hfill i\in \mathcal{V}{\setminus} \mathcal{K} 
	\end{cases}
	\label{eqModel1LeaderEquiv}
	\end{align}
	which is equivalent to \eqref{eqModel1Leader} with $\tilde\alpha_i = \textstyle {\alpha_i}/{(\alpha_i + \sum_{j\in \mathcal{N}_i}w_{ij})}$ and $\tilde{w}_{ij} = {w_{ij}}/{\sum_{j\in \mathcal{N}_i}w_{ij}}.$ 
	The works \cite{Fardad13,Borkar15} then consider 
	\begin{equation}
	\textstyle{ \min_{\mathcal{K} \subset \mathcal{V}}} \{\tilde{f}(\mathcal{K}) := \1\T(I - D_{\mathcal{K}}\tilde{W})^{-1} \1 ~|~ |\mathcal{K}| \le K\}, \label{eqfKEquiv}
	\end{equation}
	where $D_{\mathcal{K}} \!=\! I \!-\! \diag(\tilde{\boldsymbol{\alpha}}_{\mathcal{K}})$, $\tilde{\boldsymbol{\alpha}} \!=\! \1$, $\tilde{W} \!=\! [\tilde{w}_{ij}]$, and $\tilde{f}(\mathcal{K})$ represents the cumulative errors over time of all the agents, which in fact corresponds to a special case of $J^{(1)}_{\mathcal{K}}$ with $\bb = \1$ and $\boldsymbol{\xi}_0 =\1$. Specifically, \cite{Fardad13} uses a continuous relaxation of $\tilde{f}$ and the $\ell_1$-norm regularization technique to obtain a simpler optimization problem that is \emph{element-wise convex}. This allows the use of the coordinate descent approach. 
	However, it is important to point out that  the relaxed problem formulated in \cite{Fardad13} is not necessarily convex. 
	In \cite{Borkar15}, \eqref{eqfKEquiv} is used as an alternative for the problem of choosing $K$ leaders to maximize the convergence rate of \eqref{eqModel1LeaderEquiv}, which is hard to solve; the authors show supermodularity of $\tilde{f}(\mathcal{K})$ and then apply the greedy heuristic \cite{Nemhauser78} to yield approximate solutions with provable accuracy. 
	
	In \cite{Clark14AC}, the authors use a  continuous-time version of the DeGroot model and 
	consider the problem of selecting a set of nodes to become leaders (instantaneously) so as to minimize the convergence error, defined as the $l_p$-norm of the distance between the followers' states and the convex hull of the leader states. By replacing the convergence error with an upper bound that is \textit{independent} of the initial states of the network (and is  loose in general), \cite{Clark14AC} proves supermodularity of the so-obtained bound based on a connection with random walk theory, and then employs the greedy approach of \cite{Nemhauser78}.
	

	\subsubsection{Multiple leaders case}
	In \cite{Yildiz13ACM}, the authors consider a linear stochastic model the mean behavior of which is equivalent to the following deterministic model: for any $t\ge 0$, $x_i(t) = 0$ if $i\in \mathcal{V}_0$, $x_i(t) = 1$ if $i\in \mathcal{V}_1$ and 
	\begin{align}
	\textstyle x_i(t+1) =	\sum_{j\in \mathcal{V}} \tilde{w}_{ij}x_j(t), \quad i\in \mathcal{V} \setminus (\mathcal{V}_0 \cup \mathcal{V}_1),
	\label{eqSysYildiz11}
	\end{align}
	where $\mathcal{V}_0, \mathcal{V}_1 \subset \mathcal{V}$ are two disjoint sets representing two types of stubborn agents. This model is a limiting case of \eqref{eqModel2Leader} with $\alpha_i,\beta_i \in \{0,\infty\}$. 
	The optimal stubborn agent placement  problem studied in \cite{Yildiz13ACM} is stated as follows: For a given set $\mathcal{V}_0$ with known locations, choose $K$ nodes from $\mathcal{V}{\setminus} \mathcal{V}_0$ to form $\mathcal{V}_1$ so that the network bias toward $\mathcal{V}_1$ in the limit is maximized. 
	This problem is a special case of (P2) with $\bb=\1$ and $\alpha_i,\beta_i \in \{0,\infty\}$. 
	Similarly, \cite{Gionis13} 
	considers the model
	$x_i(t\!+\!1) \!=\! \sigma_i x_i(0) \!+\! (1\!-\!\sigma_{i})\textstyle \sum_{j\in \mathcal{V}} \tilde{w}_{ij}x_j(t)$, 
	where $\sigma_{i}\in [0,1]$ reflects the stubbornness level of agent $i$ regarding its initial opinion. The paper considers the problem of selecting $K$ nodes to become fully stubborn with their opinions set to $1$, then the limiting opinions of all the agents, on average, are as positive as possible. 
	In both \cite{Yildiz13ACM} and \cite{Gionis13}, submodularity of the objective functions is shown (based on connections with a random walk) and then the greedy algorithm \cite{Nemhauser78} is used to approximate the optimal solution within factor $(1-e^{-1})$.

	\subsubsection{Our Contributions}
	
	This paper generalizes and differs from the works above both in problem formulation and solution. 
	
	Regarding problem formulation, our direct followers can have dynamics like any other network node, unlike the forceful/stubborn agents in previous papers. 
	Moreover, within the context of problem (P1), the agents' initial opinions need not be the same and are taken into account explicitly in the cost $J_{\mathcal{K}}^{(1)}$, which is a tight upper bound on the cumulative convergence error of all the agents. 
	Furthermore, the agents can be weighted differently by the leader in contributing to the cost $J_{\mathcal{K}}$. 
	We believe that these are natural settings subsuming many existing scenarios in the literature, and thus likely to be of increased value for practical applications. 
	Finally, the models considered here, i.e., \eqref{eqModel1Leader} and \eqref{eqModel2Leader}, allow us to establish the convexity of the relaxation to problem (P), while neither \eqref{eqSysYildiz11} nor \eqref{eqModel1LeaderEquiv} does so. 

	Regarding problem solving, we adopt two well known heuristic approaches, namely the convex relaxation/approximation technique and the  greedy selection strategy, but the theoretical results presented here are more general and stronger than existing results. 
	In particular, our technical contributions include establishment of the supermodularity property of the objective function in problem (P) and  the convexity of its continuous relaxation; both results are  based on the M-matrix theory, an approach completely different than those used in \cite{Clark14AC,Yildiz13ACM,Gionis13,Borkar15}. 
	First, we prove the convexity of our relaxed problem in the usual sense (instead of element-wise as done in  \cite{Fardad13}) and without assuming any kind of symmetry, which is of great benefit since it allows us to use  effective numerical algorithms such as  gradient descent and interior point methods. 
	Second, we derive a general matrix supermodularity inequality that can be used to prove supermodularity of $J_{\mathcal{K}}$ as well as another type of cost function encountered in the literature. 
	Combining the supermodularity result with the notion of curvature of a submodular function \cite{Conforti84}, we show that the standard greedy algorithm \cite{Nemhauser78} applied to (P) admits an approximation guarantee sharper than $(1-e^{-1})$. 
	In addition, we introduce an improved version of this algorithm that is able to achieve better accuracy. 
	Finally, in both approaches, we derive upper and lower bounds on the optimal value, which, when combined, provide a better analysis of the obtained approximate solutions. 
	As will be illustrated in our numerical example, the approximation ratio can be  ensured to range from $70\%$ to $100\%$ depending on the value $K$.

	\vspace{-1mm}
	\section{Convexification Approach} \label{secCVX}
	\vspace{-1mm}
	Next, we study the convexity of the continuous relation defined by $J_{\mathcal{K}}$ and discuss numerical methods to solve the relaxed  problem. 
	
	\vspace{-2mm}
	\subsection{Convexity of Relaxation}
	Consider problem (P), 
	equivalently stated as 
	\begin{align}
	\label{eqMinJK2b}
	\quad\textrm{(P)}\quad\quad	
	\begin{split}
	\textstyle{\min_{\sB \in \mathbb{R}^N}} \quad & f(\sB) := \bb\T (L_\beta+ \diag(\sB\diamond \boldsymbol\alpha ))^{-1}\cc\\
	\mathrm{s.t.} \quad & s_i \in \{0,1\}\quad \forall i=1,\ldots, N  \\
	& \card(\sB) \le K, 
	\end{split}
	\end{align} 
	with  
	the optimal value denoted by $f^*_P$.
	Recall that $L_\beta \!=\! L \!+\! \diag(\boldsymbol\beta)$. 
	We will also use $L_0 \!:=\! L$ to signify the case  $\boldsymbol{\beta} \!=\! \0$, i.e., problem~(P1). 
	
	Problem (P) is combinatoric (hence nonconvex) and generally hard to solve especially for large networks. 
	We defer our discussion on the properties of $f$ for now, and instead begin by discussing techniques to handle the cardinality constraint. The first idea is to relax this constraint, resulting in a continuous relaxation of (P) as follows:
	\begin{align}
	\label{eqMinJK2RelaxLB}
	\textrm{(P$\_$Rlxd)}\qquad
	\textstyle{\min_{\y}} \quad & \{f(\y)~|~ \y \in [0,1]^N, \1\T\y \le K\},
	\end{align}
	where the optimal value, denoted by $f^*_{P\_Rlxd}$, is clearly a lower bound for that of  (P), i.e., $f^*_{P\_Rlxd} \le f^*_P$. 
	Of course this bound is useful if an optimal solution $\y_{P\_Rlxd}$ is computable. In that case, if $\y_{P\_Rlxd}$ is a binary vector, then it is also the optimal for (P). 
	However, a binary solution is not to be expected as $\y_{P\_Rlxd}$ tends to be fractional.  
	In general, we can use a simple projection onto the feasible set of problem (P) to obtain an approximation (e.g., rounding up to $1$ the $K$ largest elements of $\y_{P\_Rlxd}$ and zeroing out the rest), 
	resulting in  an upper bound on $f^*_P$, which we denote by $\bar{f}_{P\_Rlxd}$. 
	
	Another practical approximation is to use $\ell_1$-norm regularization:
	\begin{align}
	\label{eqMinJK2Relax2}
	\text{(P$\_$Aprx)}~~
	\textstyle{\min} ~ & \{g(\y) \!:= f(\y) +\! \gamma \1\T \y ~|~ \y \in \Omega:= [0,1]^N \},
	\end{align}
	where $\gamma$ is a positive parameter the role of which is to promote sparsity of the solution. 
	If $\gamma = 0$, then $\y = \1$ is the global solution to this problem (see also Theorem \ref{thmConvexity_G} below); increasing $\gamma$ is a way to penalize the number of nonzero elements in the solution. 
	Let $\sB^*_{P\_Aprx}$ be the binary vector corresponding to the $K$ largest elements of a solution to (P$\_$Aprx). Then $f_{P\_Aprx} := f(\sB^*_{P\_Aprx})$ is  an upper bound on the optimal value of (P), and the gap $(f_{P\_Aprx}-f^*_{P\_Rlxd})$ can also be used to evaluate the quality of our approximations.

	We now establish the convexity of $f$, which would clearly be pertinent for problems (P$\_$Rlxd) and (P$\_$Aprx). 
	For similar cost functions that are convex under symmetry of the Laplacian matrix $L$, see, e.g., \cite{Patterson10CDC,Long10,Lin11}. 
	Here, we do not assume any symmetry conditions on the Laplacian matrix $L$ (even on its structure), or on the nonnegative vectors $\bb$ and $\cc$ (trivial cases such as $\bb = \0$ or $\cc = \0$ are excluded). The convexity proof relies on the following lemma.

	\begin{lem} \label{lemPVPVP}
		Let $A \in \mathbb{R}^{N\times N}$ be nonnegative and $V  \in \mathbb{R}^{N\times N}$ be diagonal. Then for each $m\ge 0$,  $\textstyle \sum_{i+j+k=m}A^iVA^jVA^k$ is a nonnegative matrix, where $i,j,k$ are nonnegative integers. 
	\end{lem}
	\begin{proof}
		By change of variables, we have
		$$\textstyle \sum_{i+j+k=m}A^iVA^jVA^k = \sum_{0\le q \le r \le m} A^qVA^{r-q}VA^{m-r}.$$
		Let $V=\diag([v_1,\ldots,v_N])$. The $sp$-th entry of the matrix above is 
		\begin{align}
		\textstyle \sum_{0\le q \le r \le m} \sum_{1\le i,j\le N}  [A^q]_{si}[A^{r-q}]_{ij}[A^{m-r}]_{jp} v_iv_j.\label{eqLem1a}
		\end{align}
		To simplify this expression, denote $A=[a_{kl}]$ and 
		consider the graph generated by matrix $A$ where $a_{ij}$ denotes the weight of the directed edge $i\to j$. 
		Let $\mathcal{P}_m$ denote the set of all walks of length $m$ 
		from node $s_0 = s$ to $s_m=p$, i.e., those of the form
		$s=s_0\stackrel{e_1}{\to}s_1 \stackrel{e_2}{\to} \ldots \stackrel{e_m}{\to}s_m = p,$ 
		where $e_i = (s_{i-1}s_i)$ denotes the directed edge from $s_{i-1}$ to $s_i$. 
		Now for each tuple $(qirj)$, let $P_{(qirj)} \subset \mathcal{P}_m$ denote the set of walks satisfying $s_q = i$ and  $s_r=j$ (i.e., fixing positions $q$ and $r$). Then the term under the double summation in \eqref{eqLem1a} represents the total weight of all the walks\footnote{The weight of a walk is defined as the product of the weights of all the edges along the walk.} in $P_{(qirj)}$ multiplied by $v_{s_q}v_{s_r}$, i.e.,
		\begin{align}
		[A^q]_{si}[A^{r-q}]_{ij}[A^{m-r}]_{jp} v_iv_j =\!\!\!\!\! \sum_{\{e_k\}_1^m \in P_{(qirj)}}  \!\!\!\!\!  a_{e_1}a_{e_2}\ldots a_{e_m} v_{s_q}v_{s_r}. \nnb
		\end{align}
		Summing the right side of this relation over $1\le i,j\le N$ yields the total weight of all the walks in $\mathcal{P}_m$ (each being scaled by $v_{s_q}v_{s_r}$), namely,
		$\sum_{1\le i,j\le N}\sum_{\{e_k\}_1^m \in P_{(qirj)}}  \!\!\!\!\!  a_{e_1}a_{e_2}\ldots a_{e_m} v_{s_q}v_{s_r} =  \sum_{\{e_k\}_1^m \in \mathcal{P}_{m}}  a_{e_1}a_{e_2}\ldots a_{e_m} v_{s_q}v_{s_r}.$ 
		Thus, \eqref{eqLem1a} equals
		\begin{align}
		&\textstyle \sum_{0\le q \le r \le m} \sum_{\{e_k\}_1^m \in \mathcal{P}_{m}}  a_{e_1}a_{e_2}\ldots a_{e_m} v_{s_q}v_{s_r} \nnb\\
		&= \textstyle \sum_{\{e_k\}_1^m \in \mathcal{P}_{m}}  a_{e_1}a_{e_2}\ldots a_{e_m} \sum_{0\le q \le r \le m} v_{s_q}v_{s_r}\nnb\\ 
		&= \textstyle \sum_{\{e_k\}_1^m \in \mathcal{P}_{m}} \frac{1}{2} a_{e_1}a_{e_2}\ldots a_{e_m} \big[ \big(\sum_{0\le i \le m} v_{s_i}\big)^2 \!+\!  \sum_{0\le i \le m}v_{s_i}^2  \big]\nnb
		\end{align}	
		which is nonnegative, thereby completing the proof.
	\end{proof}

	We are now ready to establish the convexity as well as other important properties of our objective functions. 
	\begin{thm} \emph{(Properties of $f$)}\label{thmConvexity_G}
		For any $\bb, \cc, \boldsymbol{\alpha} \in \mathbb{R}^N_+{\setminus}\{\0\}$ and $\boldsymbol{\beta} \in \mathbb{R}^N_+$, 
		let $\Omega= [0,1]^N$ 
		and consider $f: \mathbb{R}^N_+ \to \mathbb{R}\cup \{\infty \}$ defined in \eqref{eqMinJK2b}. 
		Then $f$ is positive, convex and decreasing on $\Omega$. It is smooth on the interior of $\Omega$ with gradient $\nabla f$ and Hessian $\mathbf{H}$ given by
		\begin{equation}
		\begin{split}
		\!\nabla f(\y) \!=\! - ( Y\nT \bb  )\! \diamond \boldsymbol\alpha\diamond  ( Y^{-1}\! \cc), ~~~Y\!:=L_\beta+ \diag(\y\!\diamond\! \boldsymbol\alpha) 
		\end{split}
		\label{eqGradG}
		\end{equation}		
		\begin{equation}
		\begin{split}
		\mathbf{H}(\y) &= H(\y)+H\T(\y) \quad \textrm{with}\quad \\
		H(\y) &:= \diag(\boldsymbol\alpha \diamond (Y\nT \bb)) Y^{-1}  \diag( \boldsymbol\alpha \diamond (Y^{-1} \cc)). 
		\end{split} \label{eqHessianG}
		\end{equation}
		Moreover, $\mathbf{H}(\y)$ is a nonnegative matrix  and 
		\begin{equation}
		0 \preceq \mathbf{H}(\y) \preceq L_f I, \quad \textrm{with} \quad L_f := \rho(\mathbf{H}(\0)). \label{eqStrongCVX}
		\end{equation}
		Furthermore, $L_f \le N\max_{ij}[\mathbf{H}(\0)]_{ij}$. 
	\end{thm}
	\begin{proof}
		Smoothness of $f$ follows from its definition. 
		Positiveness follows from assumptions $\bb, \cc, \boldsymbol{\beta} \ge \0$ and the fact that $Y = L_\beta+ \diag(\y \diamond \boldsymbol\alpha)$  is a nonsingular M-matrix whenever $\y \in \Omega$ and $\boldsymbol{\beta} $ are not both equal $\0$, which ensures that $Y^{-1}$ is a nonnegative matrix (see Lemma~\ref{thm_inverseMmatrix} in Appendix~\ref{secMatrixResults}). 
		Hence $f(\y)  = \bb\T Y^{-1}\cc \ge 0$ for all $\y \in \Omega$. Next, we find the first differential of $f$, namely, 
		\begin{align}
		\ds f(\y)  &= \bb\T \ds Y^{-1}\cc 
		= -\bb\T Y^{-1} \diag(\boldsymbol\alpha) \diag(Y^{-1}\cc)\ds \y, \nnb\\
		&= -\big[ (Y\nT \bb)\diamond \boldsymbol\alpha \diamond(Y^{-1}\cc) \big]\T\ds \y,
		\end{align}
		where we have used the fact that $\ds Y^{-1} = -Y^{-1} (\ds Y) Y^{-1}$, $\ds Y = \ds (L_\beta + \diag(\y\diamond \boldsymbol\alpha) ) = \diag(\ds \y\diamond \boldsymbol\alpha)$, and $\diag (\x)\y = \diag (\y)\x = \x\diamond\y$. Therefore, $\nabla f(\y) = -\big( Y\nT \bb \big)\diamond \boldsymbol\alpha \diamond \big( Y^{-1} \cc \big).$ 
		
		Since $Y^{-1} \ge \0$, we have $\nabla f(\y) \le \0$, which implies that $f$ is decreasing in $\y$. 
		In fact, a stronger statement holds, that is, $Y^{-1} = (L_\beta + \diag(\y\diamond \boldsymbol\alpha))^{-1}$ is nonnegative and decreasing in $\y$. As a result, $\|\nabla f(\y)\|_2 \le \|\nabla f(\0)\|_2, \forall \y \in \Omega$. When $\boldsymbol{\beta} \neq \0$, $\|\nabla f(\0)\|_2 < \infty$, thus $f$ is Lipschitz continuous with parameter $\|\nabla f(\0)\|_2$ on $\Omega$.

		Next, we find the second differential of $f$ as follows: 
		\begin{align}
		\ds^2 f(\y)  
		&= 2\bb\T Y^{-1} \diag(\ds \y \diamond  \boldsymbol\alpha) Y^{-1}  \diag(\ds \y \diamond \boldsymbol\alpha) Y^{-1} \cc,\nnb\\
		&= 2\ds \y\T \diag(\boldsymbol\alpha \diamond (Y\nT \bb)) Y^{-1}  \diag( \boldsymbol\alpha \diamond (Y^{-1} \cc)) \ds \y, \nnb\\
		&= \ds \y\T (H+H\T) \ds \y \label{eqDifferentialG}
		\end{align}
		with $H$ defined as in \eqref{eqHessianG}. 
		Thus, $\mathbf{H} = (H+H\T)$ is the Hessian of $f$. 
		Clearly, $\mathbf{H}\ge \0$ since $Y^{-1}, W, \bb, \cc$ are so. 
		
		For convexity, it suffices to show that $\ds^2 f$ given by \eqref{eqDifferentialG} is positive semidefinite on $\Omega{\setminus}\{\0\}$. 
		Indeed, since $\bb$ and $\cc$ are nonnegative, we will prove that $Y^{-1}VY^{-1}VY^{-1} \ge \0$ where  $V=\diag(\ds \y\diamond \boldsymbol\alpha)$. Note that $Y$ is a nonsingular M-matrix. Thus, by definition, $Y = s(I - A)$ for some positive $s$ and some nonnegative matrix $A$ with $\rho(A)<1$. Then we have $Y^{-1}=s^{-1}\sum_{i=0}^{\infty}A^i$ 
		and hence
		\begin{align}
		Y^{-1}VY^{-1}VY^{-1} &= \textstyle s^{-3}\sum_{i\ge 0}\sum_{j\ge 0}\sum_{k\ge 0}A^iVA^jVA^k \nnb\\
		&= \textstyle s^{-3}\sum_{m\ge 0}\sum_{i+j+k=m}A^iVA^jVA^k. \label{eqPVPVP}
		\end{align}
		Now by Lemma \ref{lemPVPVP}, 
		$\sum_{i+j+k=m}A^iVA^jVA^k \ge \0$ for any $m \ge 0$. 
		Therefore,  $Y^{-1}VY^{-1}VY^{-1} \ge \0$, thereby proving convexity of $f$. 
		
		Next, to prove \eqref{eqStrongCVX}, we use the inequality 
		\begin{equation}
		\x\T\mathbf{H}(\y)\x \le \rho(\mathbf{H}(\y))\x\T\x, \quad \forall \x\in \mathbb{R}^N, \y \in  \Omega, \label{eqHessianStrictPosG}
		\end{equation}
		which holds since $\rho(\mathbf{H}(\y))$ is the largest eigenvalue of  the nonnegative (and symmetric) matrix $\mathbf{H}(\y)$ (see Lemma~\ref{thmSpectralRadiusNonnegativeMat} in Appendix~\ref{secMatrixResults}). 
		Note also that $H(\boldsymbol{\y})$ is decreasing in $\y \in \Omega$. Thus we have 
		$\0_{N \times N} \le \mathbf{H}(\y) \le\mathbf{H}(\0) \le \max_{ij}[\mathbf{H}(\0)]_{ij} \1\1\T.$ 
		Finally, by Lemma~\ref{thmSpectralRadiusBounds} in Appendix~\ref{secMatrixResults}, we have $\rho(\mathbf{H}(\y)) \le \rho(\mathbf{H}(\0)) = L_f \le \max_{ij}[\mathbf{H}(\0)]_{ij} \rho(\1\1\T) = N\max_{ij}[\mathbf{H}(\0)]_{ij}$. 
	\end{proof}
	%
	%
	%
	%
	The following result is immediate, so the proof is omitted.  
	\begin{cor} \label{corGradient_F} 
		The function $g$ is smooth and convex on $\Omega$ with gradient $\nabla g(\y) = \nabla f(\y) + \gamma\1,$ 
		which is Lipschitz continuous with Lipschitz constant $L_g = L_f$. 
		Moreover, if $\eta := \min_{\y\in \Omega}\lambda_{\min}(\mathbf{H(y)}) >0$, then $g$ is strongly convex with parameter $\eta$. 
	\end{cor} 
	Note that when $\boldsymbol{\beta} {=} \0$ we have $L_{\beta} {=} L$, which is singular. Thus, the Lipschitz constant $L_f\! =\! \rho(\mathbf{H}(\0)) \!=\! \infty$. It is now clear that both (P$\_$Rlxd) and (P$\_$Aprx) are convex with a (possibly strongly) convex smooth cost function and thus can be solved by various algorithms, including Interior Point Methods (IPMs) and the Projected Gradient Method (PGM) (see e.g., \cite{Wright97, Bertsekas99Book, Potra00, Nesterov04}), provided that $\nabla f(\y)$ can be evaluated efficiently (see Remark \ref{remGradientEva} below).

	\begin{rem} (\emph{On selecting regularization parameter $\gamma$})\label{remRegularizationParam}
		From the optimal solution $\tilde{\y}^*$ of problem \eqref{eqMinJK2Relax2} for a particular $\gamma$, we can obtain an approximate solution to the original problem \eqref{eqMinJK2b} by choosing nodes corresponding to the $K$ largest entries of $\tilde{\y}^*$. As $\gamma$ increases, there (usually) exists $\bar\gamma$ such that $\card(\tilde{\y}^*) \le K$. Once this value is found (which can be done fairly easily), $\gamma$ can be tuned within the interval $[0,\bar\gamma]$ to find the best approximation. 
	\end{rem}

	\vspace{-2mm}
	\subsection{Numerical Methods}
	We briefly discuss two numerical algorithms that can be used to solve  problem (P$\_$Aprx). 
	Problem (P$\_$Rlxd) can be treated similarly.
	
	First, for not very large networks, we can use primal-dual IPMs \cite{Wright97}, where each iteration involves computing the Newton direction, which requires $O(N^3)$ operations to evaluate gradient $\nabla f$ and Hessian matrix $\mathbf{H}$, given respectively in \eqref{eqGradG} and \eqref{eqHessianG}. The storing cost is $O(N^2)$. 
	In practice, the method converges in a few iterations. 
	
	Second, for large networks where IPMs are not suitable, we can use 
	the PGM (only requiring gradient evaluations) given by:
	\begin{equation}
	\y^{(t+1)} = P_{\Omega} \big[ \y^{(t)} - \mu^{(t)} \big( \nabla f(\y^{(t)}) + \gamma\1 \big) \big], \label{PGM}
	\end{equation}
	where $P_{\Omega}$ denotes the projection operator onto  $\Omega$  and  $\mu^{(t)}$ step size. 
	It follows from \cite[Prop.~2.3.1 and 2.3.2]{Bertsekas99Book} and \cite[Thm.~2.2.8]{Nesterov04} that (i) if $\mu^{(t)}$ is chosen by the Armijo rule, then every limit point of $\{\y^{(t)}\}$ is an optimal solution to problem \eqref{eqMinJK2Relax2}, and (ii) if $\boldsymbol{\beta} \neq \0$, any constant step size  $\mu^{(t)} \equiv \mu \in (0, L_f)$ can be used. Moreover, if $\eta > 0$, then for $\mu = {L_f^{-1}}$, $\y^{(t)}$ converges linearly to the unique solution $\y^*$ with  rate $\sqrt{1-{\eta}{\mu}}$.
	
	Note also that when $\boldsymbol{\beta} \!=\! \0$ (i.e., problem (P1) where $f(\0)=\infty$), iteration \eqref{PGM} should be modified as follows. 
	Given any $\y^{(0)} \in \Omega\setminus \{\0\}$, $\Omega^{0} = \{\y \in \Omega| g(\y) \le g(\y^{(0)})\}$ is a convex compact set excluding $\0$. Thus, $g$,  $\nabla g$ and $\nabla^2 g = \mathbf{H}$ are continuous on $\Omega^{0}$. Moreover, $\nabla g$ is Lipschitz continuous on $\Omega^0$ with coefficient $L_g^0 = \max_{\y \in \Omega^0} \|\mathbf{H}(\y)\|_2$.  As a result, we can replace $P_\Omega$ by $P_{\Omega^0}$ or choose a step size such that $\y^{(t)} \in \Omega^{0}$.

	\begin{rem}\label{remGradientEva}(\textit{On gradient evaluation}) 
		Computing $\nabla f(\y)$ involves inversion of $Y = (L_\beta+ \diag(\y\diamond \boldsymbol\alpha))$, which usually costs $O(N^3)$ operations 
		and $O(N^2)$ memory storage, and thus is not practical for large networks. 
		In such a case, we can resort to the following alternative. From \eqref{eqGradG}, we have $\nabla f(\y) \!=\! -\ub\! \diamond\! \boldsymbol\alpha \!\diamond\! \vb$  with  
		$\ub= Y\nT \bb$ and $ \vb = Y^{-1} \cc.$ 
		So, $\uu$ and $\vv$ are respectively the solutions to the sparse linear equations $Y\T \ub= \bb$ and $Y\vb = \cc$, 
		for which many algorithms are available, e.g., power-iteration. Specifically, let $Y = D_y + E$ where $D_y$ and $E$ denote the diagonal and off-diagonal parts of $Y$. 
		Clearly, only $D_y$ depends on $\y$. 
		Now consider $\ub$, 
		which satisfies 
		$
		\bb =  D_y\ub + E\T\ub. 
		$
		Since $D_y$ is invertible, we have a fixed point relation 
		$
		\ub = -D_y^{-1}E\T\ub+D_y^{-1}\bb. 
		$
		Under Assumptions \ref{asmp_fixedGraph} and \ref{asmp_alpha}, the right side defines a contraction mapping with coefficient $\rho(D_y^{-1}E\T) < 1$. 
		Thus, we can use iteration $\ub_{k+1} = -D_y^{-1}(E\T\ub_k-\bb)$ 
		to compute $\ub$, which is highly scalable since (i) $E$ is sparse and can be read off from $L$ (or $W$), whose storage takes only $O(|\mathcal{E}|)$, 
		and (ii) the computation also takes $O(|\mathcal{E}|)$ operations. 
		Moreover, suppose we terminate this iteration in $k_u$ iterations, 
		with convergence error proportional to $\rho^{k_u}(D_y^{-1}E\T)$, 
		then the running time to compute $\ub$ is $O(k_u|\mathcal{E}|)$. 
		Finally, $\vb$ can be computed in the same manner, namely, $\vb_{k+1} = -D_y^{-1}(E\vb_k-\cc)$. 
	\end{rem}

	\vspace{-2mm}
	\section{Supermodularity and Greedy Algorithms} \label{secSupermodularity}
	In this section, we develop an alternative approach to problem (P) based on the greedy strategy where approximation bounds for the suboptimal solutions can be established. 
	To this end, we first prove that $J_{\mathcal{K}}$ is monotone and supermodular in the set-variable $\mathcal{K}$. 
	As a result, problem (P) 
	admits an accuracy $(1 - e^{-1})$ approximation algorithm \cite{Nemhauser78}. 
	We then develop an improved version of this algorithm that can achieve better approximate solutions. 
	
	\vspace{-2mm}
	\subsection{Supermodularity Results} 
	Our main results in this subsection are the following two lemmas, the first of which is a \textit{matrix supermodularity inequality} and the second is a \textit{composition property}.  
	\begin{lem} \label{lemMatrixSupermodularity} 
		For $\mathcal{S} \subset \mathcal{V}$, let $\Gamma_\mathcal{S} = \diag(\boldsymbol{\alpha}_{\mathcal{S}})$. The function $(L_\beta+ \Gamma_\mathcal{S})^{-1} \in \mathbb{R}^{N\times N}_+$ is nonincreasing and supermodular in $\mathcal{S}$, i.e., the following matrix inequalities hold for any $v, k \in \mathcal{V}{\setminus}\mathcal{S}$
		\begin{equation}
		\begin{split}
		&(L_\beta+ \Gamma_\mathcal{S})^{-1} - (L_\beta+ \Gamma_{\mathcal{S} \cup \{v\}})^{-1} \\
		&\geq (L_\beta+ \Gamma_{\mathcal{S}\cup \{k\}})^{-1}	- (L_\beta+ \Gamma_{\mathcal{S} \cup \{k,v\}})^{-1} \ge \0. 
		\end{split}
		\label{eqMatrixSupermodular}
		\end{equation}
		This result also holds true if we replace $L_\beta$ with $L_0$. 
	\end{lem}
	\begin{lem} \label{lemComposition} 
		If $F\!:\!2^\mathcal{V} \!\to\! \mathbb{R}^{N\times N}$ is decreasing and supermodular, and 
		$f\!:\!\mathbb{R}^{N\times N}\!\! \to\! \mathbb{R}$ is increasing and convex, then the composition $f\circ F$ is nonincreasing and supermodular.
	\end{lem}
	
	The proofs of these lemmas are given in Appendices~\ref{proofMatrixSupermodularity} and \ref{proofComposition} respectively. The case $\mathcal{S} \!=\! \varnothing$ is included in Lemma~\ref{lemMatrixSupermodularity}  since $(L_\beta\!+\! \Gamma_{\varnothing})^{-1} \!=\! +\infty$ if $\boldsymbol{\beta} \!=\! \0$. 
	Now applying Lemmas \ref{lemMatrixSupermodularity} and \ref{lemComposition} with $F(\mathcal{K}) \!=\! (L_\beta \!+\! \Gamma_\mathcal{K})^{-1}$ and $f(X) {=} \bb\T X\cc$ yields the following: 
	\begin{thm}\label{thmSupermodularJK}
		$J_{\mathcal{K}}$ is supermodular and nonincreasing in $\mathcal{K}$. 
	\end{thm}

	
	Note that \cite{Lin11} considers the problem of selecting a number of agents as leaders so as to minimize the overall variance in an undirected unweighted network subject to stochastic disturbances. It can be verified that the cost function therein is equivalent to $\tr\big( (L + \diag(\boldsymbol\alpha_{\mathcal{K}} ))^{-1} \big)$, which is $(f\circ F)(\mathcal{K})$ with $F(\mathcal{K}) = (L+ \Gamma_\mathcal{K})^{-1}$ and  $f(X) = \tr(X)$. Thus,  we can immediately conclude  supermodularity of this cost function; this  was not shown in \cite{Lin11}.
	
	
	\vspace{-2mm}
	\subsection{Greedy Algorithms and Ratio Bounds}
	Having established supermodularity of $J_{\mathcal{K}}$, we now introduce our greedy algorithms and show their ratio bounds. For convenience, $J_{\mathcal{K}}$ and $J(\mathcal{K})$ are used interchangeably. 
	Our Algorithm~\ref{alg_HeuSupPhase1}, whose output is denoted by $\mathcal{K}^{G}$, is based on the greedy algorithm in \cite{Nemhauser78}. 

	\begin{algorithm2e}[ht]
		\DontPrintSemicolon
		$\text{Init: }\mathcal{K}^{G} \leftarrow \varnothing$\;
		\For{$i = 1:K$}{
			$k_i^* \gets \arg\min_{v\not\in\mathcal{K}^{G}} J({ \mathcal{K}^{G}\cup \{v\} })$\;
			$\mathcal{K}^{G} \gets \mathcal{K}^{G} \cup \{k_i^*\}$\;
		}
		\text{Return: }$\mathcal{K}^{G}$\;
		\caption{\emph{Greedy Adding} $\mathcal{K}^{G}$}\label{alg_HeuSupPhase1}
	\end{algorithm2e}
	\vspace{-1mm}


	\begin{rem}\label{remComplexityAlg1}\emph{(Complexity of Alg.~\ref{alg_HeuSupPhase1})} 
		Without exploiting the structure of $J_\mathcal{K}$, Algorithm~1 requires $O(N^2)$ memory and $O(KN^4)$ operations (due to  matrix inversion). We can use \textit{Rank-1 updates} or \textit{power-iteration method} to alleviate this burden. In particular, at any iteration, let $\mathcal{S}$ denote the current set $\mathcal{K}^{G}$ and let $P {=} (L_\beta{+} \Gamma_\mathcal{S})^{-1}$. By the Woodbury identity (see Lemma~\ref{thmWoodbury} in Appendix~A), 
		$(L_\beta{+} \Gamma_{\mathcal{S} \cup \{v\}})^{-1} {=} P {-} {P_{(v)}P^{(v)}}{ (\alpha_v^{-1} {+} P_{vv})^{-1} }$ is a rank-1 update from $P$. 
		Then $J_{\mathcal{S}} \!-\! J_{\mathcal{S} \cup  \{v\}} \!=\! {\bb\T P_{(v)}P^{(v)}\cc}/{ (\alpha_v^{-1} \!+\! P_{vv})} \!=:\! \Delta J(v,\mathcal{S}).$
		Thus, knowing $P$, it requires $O(N)$ operations to find $\Delta J(v,\mathcal{S})$ and hence $O(N^2 {-} N|\mathcal{S}|)$  to find $v^* = \arg\max_{v\in \mathcal{V}{\setminus}\mathcal{S}} \Delta J(v,\mathcal{S})$. Then $(L_\beta \!+\! \Gamma_{\mathcal{S} \cup \{v^*\}})^{-1}$ is then obtained from $P$ by a rank-1 update, taking $O(N^2)$. 
		The initial case $\mathcal{S}=\varnothing$ corresponds to $P = L_{\beta}^{-1}$ (or pseudo-inverse of $L$ if $\boldsymbol{\beta}=\0$) requiring $O(N^3)$ operations. 
		To sum up, using this scheme, the algorithm requires $O(KN^2+N^3)$ operations and $O(N^2)$ memory space. 
		For large networks, 
		we can exploit the sparsity structure of $L_\beta$ in connection with the power-iteration method as shown in Remark~\ref{remGradientEva}. 
	\end{rem}

	Note that in our previous work \cite{Mai16a}, the same greedy algorithm using rank-1 updates has been applied for the case of problem (P1). 
	In this paper, we use this algorithm for (P) with more general setting and provide proofs of the supermodularity of $J_{\mathcal{K}}$ and the ratio bounding the error incurred, which were not included in \cite{Mai16a}. We now provide the approximation ratio of Algorithm \ref{alg_HeuSupPhase1}. 
	Let $Z(\mathcal{S})$ be nondecreasing submodular in $\mathcal{S}$. The curvature of $Z$ with respect to a set $\mathcal{P}$ is defined as (see, e.g., \cite{Conforti84}) $\sigma := \textstyle 1-\min_{x\in \mathcal{P}}\frac{Z(\mathcal{P}{\setminus}\{x\}) - Z(\mathcal{P}) }{Z(\varnothing) - Z(\{x\})}.$
	
	\begin{thm}(\cite[Cor.~5.7]{Conforti84})\label{thmSubmodularBound}
		Let $Z(\mathcal{S})$ be nondecreasing submodular in $\mathcal{S}$ with $Z(\varnothing) = 0$. Let $\mathcal{S}^G$ and $\mathcal{S}^*$ denote the greedy solution and an optimal one to problem $\max \{Z(\mathcal{S}): \mathcal{S} \subseteq \mathcal{P}, |\mathcal{S}|\le K \}$. Then  
		\begin{align}
		\textstyle {Z(\mathcal{S}^G)}/{Z(\mathcal{S}^*)} \ge 
		\frac{1}{\sigma}\big( 1-(1-\frac{\sigma}{K})^K \big) =: R_{\sigma,K}, \label{eqGreedyBoundCurvature}
		\end{align}
		where $\sigma$ is the curvature of $Z$ with respect to $\mathcal{P}$.
	\end{thm}

	\begin{thm}\label{thmRatioBoundAlg1}
		Let Assumptions \ref{asmp_fixedGraph} and \ref{asmp_alpha} hold. Let $\mathcal{K}^*$ denote an optimal solution to (P) and  let $\mathcal{K}^{G}$ be the output of Algorithm~\ref{alg_HeuSupPhase1}. 
		\begin{itemize}
			\item[(i)] 	Let $v^* = \arg\min_{v\in \mathcal{V}_{\alpha}} J(\{v\})$. If $\boldsymbol\beta = \0$, then 
			\begin{align}
			{J(\{v^*\}) -  J(\mathcal{K}^G)} 
			\ge R_{\sigma,K-1} {\big( J(\{v^*\})  -  J(\mathcal{K}^*)\big)},
			\label{eqRatioBoundJ1}
			\end{align}
			where $\sigma = 1-\min_{x\in \mathcal{V}_{\alpha}{\setminus} \{v^*\} } \frac{J(\mathcal{V}_{\alpha}\setminus\{x\}) - J(\mathcal{V}_{\alpha})}{J(\{v^*\}) - J(\{v^*,x\})}$.
			\item[(ii)] 	If $\boldsymbol\beta \neq \0$, then 
			\begin{align}
			{J(\varnothing) - J(\mathcal{K}^G)}
			\ge R_{\sigma,K}{\big( J(\varnothing) - J(\mathcal{K}^*) \big)},
			\label{eqRatioBoundJ2}
			\end{align}
			where $\sigma = 1-\min_{x\in \mathcal{V}_{\alpha}} \frac{J(\mathcal{V}_{\alpha}\setminus\{x\}) - J(\mathcal{V}_{\alpha})}{J(\varnothing) - J(\{x\})}$.
		\end{itemize} 
	\end{thm}
	\begin{proof}
		For case (i), let $Z(\mathcal{S}) \!=\! J_{\{v^*\}} {-} J_{\mathcal{S}\cup \{v^*\}}$ for any $\mathcal{S} {\subseteq} \mathcal{V}_{\alpha}{\setminus}\{v^*\}$. It can be shown that $Z$ is nondecreasing, submodular with curvature $\sigma$  and $Z(\varnothing) =0$. Applying Theorem \ref{thmSubmodularBound} and rearranging terms yield \eqref{eqRatioBoundJ1}. For (ii), we have \eqref{eqRatioBoundJ2} follows from Theorem \ref{thmSubmodularBound} with $\sigma$ being the curvature of $Z(\mathcal{S}) := J(\varnothing) - J(\mathcal{S})$ for any $\mathcal{S} \subseteq \mathcal{V}_{\alpha}$.
	\end{proof}
	
	Note that $R_{\sigma,K} >  (1-e^{-\sigma})/{\sigma} > 1-e^{-1}$ for any $\alpha \in (0,1)$ and $K\ge 1$. Thus in general  $R_{\sigma,K}$ is tighter than the constant bound $(1-e^{-1})$ established in \cite{Nemhauser78} (and also \cite{Gionis13,Clark14AC,Borkar15}). 
	
	

	In the following, we construct another algorithm  (called \textit{greedy swapping}), which contains Algorithm \ref{alg_HeuSupPhase1} as a special case and is able to practically improve accuracy. 
	The idea is still to greedily select one ``best" node at a time, but we additionally employ a particular strategy of the Interchange Heuristic \cite{Nemhauser78}: to repeatedly replace every node in $\mathcal{K}$ by another node in $\mathcal{V}{\setminus}\mathcal{K}$ (or more precisely $\mathcal{V}_{\alpha}{\setminus}\mathcal{K}$ ) if and only if the swapping \textit{results in the largest decrease} in the objective function. 
	%

	\begin{algorithm2e}[tbh]
		\DontPrintSemicolon
		\text{Init:}  $\mathcal{K}_0^{S} \subseteq \mathcal{V}_{\alpha}, |\mathcal{K}_0^{S}|\le K$\;
		\For{$m=1:M$ or until $\mathcal{K}_{m}^{S}=\mathcal{K}_{m+1}^{S}$}{
			$\mathcal{S} \gets \varnothing, \mathcal{T} =\{ t_1, t_2, \ldots\} \gets \mathcal{K}_{m-1}^{S}$\;
			\For{$i = 1:K$}
			{$\mathcal{T} \gets \mathcal{T} {\setminus} \{t_i\}$, 
				$\quad\mathcal{U} \gets \mathcal{V}\setminus\mathcal{S} \cup \mathcal{T} $\;
				$t_i^* \gets  \arg\min_{v\in \mathcal{U}}  J({ \mathcal{S} \cup \{v\} \cup \mathcal{T}  })$\;
				$\mathcal{S} \gets \mathcal{S} \cup \{t_i^*\}$\;
			}
			$\mathcal{K}_{m}^{S} \gets \mathcal{S}$\;
		}
		\text{Return: } $\mathcal{K}_M^{S}$
		\caption{\emph{Greedy Swapping} $\mathcal{K}_{M}^{S} := \mathrm{GSwap}(\mathcal{K}_{0}^{S},M)$ }\label{alg_HeuSupPhase2}
		\vspace{-0mm}
	\end{algorithm2e}
	\vspace{0mm}

	Note that \cite{Joshi09} and \cite{Lin11} also employ the interchange heuristic 
	(without supermodularity property and approximation bound)
	but swapping occurs whenever an improvement of the cost function is found, which in theory can require an exponential number of exchanges to reach a local optimizer. 
	Our algorithm tries to avoid this by swapping in the direction of the steepest descent coordinate. 

	\begin{rem}\label{remComplexityAlg2}\emph{(Complexity of Alg. \ref{alg_HeuSupPhase2})} 
		For simplicity, we use a cyclic selection scheme in revising the set $\mathcal{K}^S$ in each cycle. 
		Each cycle (other than the first) requires $(KN - K^{2})$ function evaluations. That of the first cycle depends on $|\mathcal{K}_0^S|$, but is no more than $KN - \frac{K(K-1)}{2}$. Again, 
		by using the power-iteration method, we can avoid the $O(N^2)$ memory requirement as shown in Remark \ref{remComplexityAlg1}. (For very large networks, one could select (e.g., randomly) a subset $\mathcal{U} \subsetneq \mathcal{V}{\setminus} (\mathcal{S}\cup\mathcal{T})$ of manageable sizes in step 5). For not too large networks, we can employ the Woodbury matrix identity (see Lemma~\ref{thmWoodbury} in Appendix~A) for rank-2 updates. 
		Specifically, suppose we want to check for a possible swap between $t\in \mathcal{P}:=\mathcal{T}\cup \mathcal{S}$ with some $v\in\mathcal{U}:= \mathcal{V}{\setminus}\mathcal{P}$. Let $P= (L_\beta+ \Gamma_{\mathcal{P}})^{-1}$ and $E_{(tv)}=[\e_t,\e_v]$, where $\e_t$ is the $t$-th unit vector in $\RN$. Then $(L_\beta \!+\! \Gamma_{\mathcal{P}{\setminus}\{t\} \cup \{v\}})^{-1}$ equals
		\begin{align}
		& \textstyle  P \!-\! PE_{(tv)}\begin{bmatrix}
		P_{tt}\!-\!\alpha_t^{-1} & P_{tv} \\
		P_{vt} & P_{vv}\!+\!\alpha_v^{-1}
		\end{bmatrix}^{-1} \! E_{(tv)}\T P. \label{eqRank2Update}
		\end{align}
		Thus, $\Delta^2 J(-t,v,\mathcal{P}) {:=} J(\mathcal{P}) {-} J(\mathcal{P}{\cup} \{v\}{\setminus}\{t\})$ can be computed as 
		$[\bb\T P_{(t)}, \bb\T P_{(v)}] \begin{bmatrix}
		P_{tt} -\alpha_t^{-1}& P_{tv} \\
		P_{vt} & P_{vv} + \alpha_v^{-1}
		\end{bmatrix}^{-1}\begin{bmatrix}
		P^{(t)}\cc\\
		P^{(v)}\cc
		\end{bmatrix}$ 
		which takes  $O(N)$ operations provided that $P$ is known. Hence, finding $v^* {=} \arg\max_{v\in \mathcal{V}{\setminus}\mathcal{P}} \Delta^2 J(-t_i,v,\mathcal{S})$ requires $O(N^2{-}NK)$ operations and if a swap is performed, then $(L_\beta+ \Gamma_{\mathcal{P}{\setminus}\{t_i\} \cup \{v^*\}})^{-1}$ is computed from $P$ by a rank-2 update \eqref{eqRank2Update}, which takes $O(N^2)$. 
		(Note that the foregoing calculation resulting in the swapping selection above is also more computationally expensive than finding a possible greedy swap; which is also one of the reasons we opt for the greedy swapping strategy instead of the swapping method used in \cite{Joshi09} and \cite{Lin11}.)
		During each cycle, at most $K$ swaps can be carried out, taking $O(KN^2)$ operations. 
		For the initial cycle, if $P$ is not supplied, then its computation costs at most $O(N^3)$. 
		Thus, in general, for $M$ cycles,  Algorithm~\ref{alg_HeuSupPhase2} takes $O(MKN^2 + N^3)$ operations.
		However, from our simulations, a good value of $M$ is usually small (say 2-3). 
	\end{rem}
	

	\begin{thm}\label{thmAlgorithm2} \emph{(Properties of Alg. \ref{alg_HeuSupPhase2})} 
		Let  $\{\mathcal{K}_m^S\}_0^M$ denote the sequence of approximate solutions generated by Algorithm~\ref{alg_HeuSupPhase2}.
		\begin{itemize}
			\item[(i)] 	If $\mathcal{K}_0^S = \varnothing$, then $\mathcal{K}_1^S \equiv \mathcal{K}^{G}$--the output of Algorithm \ref{alg_HeuSupPhase1}.  
			\item[(ii)]  For any $m\ge 0$ and $\mathcal{K}_0^S \subseteq \mathcal{V}_{\alpha}$, 
			$J(\mathcal{K}_{m+1}^S) \le J(\mathcal{K}_m^S).$ 
			In fact, let $m^*$ denote the smallest index such that $\mathcal{K}_{m^*}^S = \mathcal{K}_{m^*+1}^S$, then 
			$J(\mathcal{K}_{m}^S) > J(\mathcal{K}_{m+1}^S), \quad\forall m< m^*.$
			
			\item[(iii)] Let $v^* {=} \arg\min_{v\in \mathcal{V}_{\alpha}} J^{(1)}(\{v\})$. For any $\mathcal{K}_0^S \subseteq \mathcal{V}_{\alpha}$, we have 
			$\frac{J^{(1)}(\{v^*\}) - J^{(1)}(\mathcal{K}_{m^*}^S)}{J^{(1)}(\{v^*\}) - J^{(1)}(\mathcal{K}^*)}
			> \frac{1}{2} \text{~and~}
			\frac{1 - J^{(2)}(\mathcal{K}_{m^*}^S)}{1 - J^{(2)}(\mathcal{K}^*)}
			>\frac{1}{2}.$
		\end{itemize}
	\end{thm}
	\begin{proof}
		(i) Consider $\mathcal{K}_0^S = \varnothing$ and the first cycle, i.e., $m=1$. So, $\mathcal{T} {=} \varnothing$ and $S$ is initialized as empty. As a result, line~6 becomes: $ t_i^* = \arg\min_{v\in \mathcal{V}{\setminus}\mathcal{S}} J({ \mathcal{S} \cup \{v\} })$, which together with line~7 is the greedy algorithm \ref{alg_HeuSupPhase1}. Therefore, $\mathcal{K}_1^S \equiv \mathcal{K}^{G}$.
		
		(ii) 
		Existence of $m^*$ follows from the fact that the feasible set of $\mathcal{K}$ is finite. The rest is straightforward and thus is skipped. 
		
		(iii) For any submodular and nondecreasing function $Z(\mathcal{S})$, it follows from \cite[Thm.~5.1]{Nemhauser78} that $\frac{Z(\mathcal{S}^*) - Z(\mathcal{S}^I)}{Z(\mathcal{S}^*) - Z(\varnothing)} \le \frac{K-1}{2K-1} <\frac{1}{2},$ 
		where $\mathcal{S}^*$ and  $\mathcal{S}^I$ denote the optimal solution  and an interchange solution (i.e., no more possible local improvement) to the problem $\max \{Z(\mathcal{S}): \mathcal{S} \subseteq \mathcal{P}, |\mathcal{S}|\le K \}$. 
		Applying this result to our case, where $Z(\mathcal{S}) := J^{(1)}(\{v^*\}) - J^{(1)}(\mathcal{S}\cup \{v^*\}), \forall \mathcal{S} \subseteq \mathcal{V}_{\alpha}{\setminus}\{v^*\}$ for  (P1) or  $Z(\mathcal{S}) := 1 - J^{(2)}(\mathcal{S}), \forall \mathcal{S} \subseteq \mathcal{V}_{\alpha}$ for (P2), yields the desired results. Here, $\mathcal{K}_{m^*}^S$ is an interchange solution for each $\mathcal{K}_0^S \subseteq \mathcal{V}_{\alpha}$. 
	\end{proof}

	The ratio bound $\frac{1}{2}$ in part (iii) is less than the constant $R_{\sigma,K}$ in Theorem~\ref{thmRatioBoundAlg1} but holds for any initial set $\mathcal{K}_0^S$. 
	Note also that the first part of this proposition asserts that Algorithm \ref{alg_HeuSupPhase1} can be obtained from Algorithm \ref{alg_HeuSupPhase2} by letting $\mathcal{K}_0^S = \varnothing$ and $M = 1$. In this case, the performance of the latter algorithm is no worse than the former. In fact, it is clear from part (ii) that better estimates are  attained almost surely  when $M>1$. 
	Although we are not yet able to quantify this gain rigorously, our simulation results  illustrate radical improvement compared to Algorithm \ref{alg_HeuSupPhase1}, even with small values of $M$.

	\begin{rem}(\emph{On implementation of Alg. \ref{alg_HeuSupPhase2}}) First, the algorithm works for any choice of $\mathcal{K}_0^S$ and thus can also be used to improve upon a good starting set $\mathcal{K}_0^S$ available from, e.g., the convex relaxation approach or Algorithm \ref{alg_HeuSupPhase1}. Second, when a local minimizer $\mathcal{K}_{m^*}^S$ is found, there are practical techniques to possibly escape this local minimizer at the expense of more computation time and power; e.g., random  swapping of multiple nodes in $\mathcal{K}_{m^*}^S$ with $\mathcal{V}{\setminus} \mathcal{K}_{m^*}^S$. Finally, in our simulations even with a small $M$ (say 2-3), the algorithm still finds a good approximation, especially from a good starting point. This may be attributable to the ``diminishing returns" nature of $J(\mathcal{K})$ resulting in significant improvements only in the first few cycles.
	\end{rem}
	
	\vspace{-2mm}
	\section{Numerical Example} \label{secSimulations}

	%
	Consider a directed network based on the largest strongly connected component of the Wikipedia vote network\footnote{Data available at: http://snap.stanford.edu/data/wiki-Vote.html} 
	studied in \cite{Leskovec10}.
	Our network has $1300$ nodes and $39456$ edges. We generate the weight of each directed edge randomly in $(0,1)$. 
	Suppose that leader $Q$ has selected the set $\mathcal{V}_{\beta}$ containing the first $50$ nodes with the highest out-degrees and that $\beta_i = 10^6, \forall i\in \mathcal{V}_{\beta}$  (thus, they strongly support leader $Q$). Suppose that leader $T$ can connect to up to $K$ nodes in $\mathcal{V}_{\alpha}$, which contains the first $1000$ nodes (in terms of the numbering sequence of the nodes) that are not direct followers of $Q$. 
	We also assume that $\alpha_i = 10, \forall i\in \mathcal{V}_{\alpha}$. 
	%
	%
	%
	We consider problem (P2) for different values of $K \in [1,200]$  using various schemes\footnote{Our simulations were carried out in Matlab\textsuperscript{\textregistered} R2015b on a PC with Intel\textsuperscript{\textregistered} Core\texttrademark $i7$ CPU{@}3.10 GHz and 12 GB of RAM.} including the degree and page-rank heuristics and the following:
	\begin{itemize}
		\item[(i)]\underline{Algorithm~\ref{alg_HeuSupPhase1}}:  the greedy algorithm with output $\mathcal{K}^G$ providing $J_{GU} = J(\mathcal{K}^G)$ and $J_{GL} = 1 - \frac{1-J_{GU} }{R_{\sigma,K}}$ as upper and lower bounds on $J^*$; see \eqref{eqRatioBoundJ2}.
		\item[(ii)] \underline{(P$\_$Rlxd)+IPM}: (P$\_$Rlxd) solved by the IPM in OPTI toolbox \cite{Currie12},\footnote{Here, we let $\y^{(0)} {=} \0$ and stop the algorithm if ${|f_i \!-\! f_{i-1}|} \le 10^{-6}{|f_i|}$.} which gives $\bar{f}_{P\_Rlxd}$ and $f^*_{P\_Rlxd}$ as upper and lower bounds.
		\item[(iii)] \underline{(P$\_$Aprx)+IPM}: (P$\_$Aprx) solved by the IPM (with sparsity threshold set to $0.01$). The output, denoted by $\mathcal{K}^{P\_Aprx}$, yields corresponding cost $f_{P\_Aprx} \!=:\! J_{P\_Aprx}$, an upper bound on~$J^*$. 
		\item[(iv)] \underline{$\mathrm{GSwap}(\mathcal{K}^{P\_Aprx},1)$}: applying one cycle of the greedy swapping algorithm to $\mathcal{K}^{P\_Aprx}$ obtained from (iii). 
	\end{itemize}
	
	Fig.~\ref{fig_bounds_a10} shows the upper and lower bounds mentioned above. 
	Here, the upper bounds from the greedy algorithm and the $\mathrm{GSwap}(\mathcal{K}^{P\_Aprx},1)$ and (P$\_$Rlxd)+IPM schemes are very close, 
	while the convex relaxation approach gives the best lower bounds, which help in evaluating approximation errors. In particular, using these bounds, we are able to conclude that  the the approximation ratio of greedy solutions $\mathcal{K}^G$ (as well as that of $\mathrm{GSwap}(\mathcal{K}^{P\_Aprx},1)$ and (P$\_$Rlxd)+IPM) satisfies 
	$$\textstyle {(1-J(K^G))}/{(1-J^*)} \ge {(1-J_{GU})}/{(1-f^*_{P\_Rlxd})},$$ 
	where the right side (depicted by a dotted line in Fig.~\ref{fig_bounds_a10}), is clearly much higher than $R_{\sigma,K}$ (here $\sigma = 0.99$) and the well-known ratio $(1{-}\frac{1}{e}) = 63.21\%$ for the greedy algorithm. 
	E.g., our approximation ratio is at least $90\%$ for $K\ge 90$. 
	Regarding running time, the greedy algorithm scales almost linearly with $K\ll N$, while the convex  approach does not; see  Fig.~\ref{fig_cpu_time_a10}. As $\gamma$ increases, $|\mathcal{K}^{P\_Aprx}|$ decreases, and thus so does the running time of $\mathrm{GSwap}(\mathcal{K}^{P\_Aprx},1)$.

	\begin{figure}[t] 
		\centering
		\includegraphics[scale=0.675]{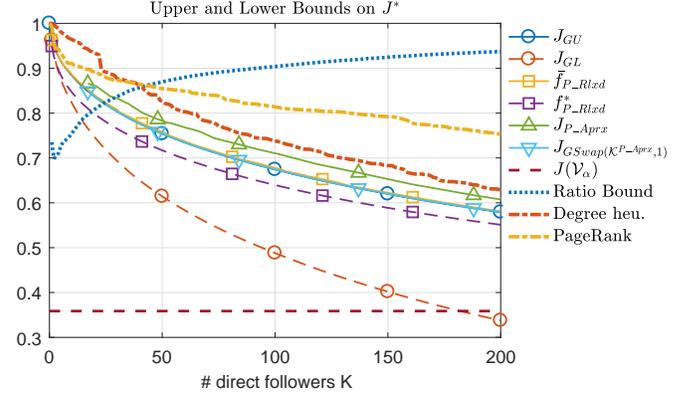}
		\caption{Upper bounds: solid and dash-dotted lines, lower bounds: dashed lines. The global lower bound $J(\mathcal{V}_{\alpha})$ holds for any $K$. The ratio bound ${(1{-}J_{GU})}/{(1{-}f^*_{P\_Rlxd})}$ (dotted line)  is at least $90\%$ as $K\ge 90$.}
		\vspace{-1mm}
		\label{fig_bounds_a10}
	\end{figure}
	\begin{figure}[t] 
		\centering
		\includegraphics[scale=0.625]{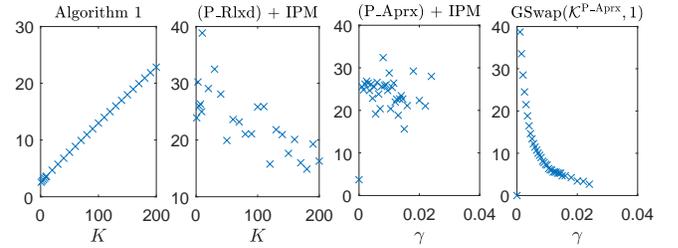}
		\vspace{-1mm}
		\caption{CPU run times ($s$). The IPM takes approximately $0.21 s$/iteration.}
		\vspace{-4mm}
		\label{fig_cpu_time_a10}
	\end{figure}

	\vspace{-1mm}
	\appendix
	\vspace{-1mm}
	\subsection{Known Matrix Results} \label{secMatrixResults}
	\vspace{-1mm}
	\begin{lem}\label{thmSpectralRadiusNonnegativeMat} \emph{(\cite[Thm. 8.3.1]{Horn85})} 
		If $A\in \Mn_+$, then $\rho(A)$ is an eigenvalue of $A$ and $\exists \x \in \mathbb{R}^N_+{\setminus} \{\0\}$ such that $A\x = \rho (A) \x.$
	\end{lem}
	
	\begin{lem}\label{thmSpectralRadiusBounds}\emph{(\cite[Thm. 8.1.18]{Horn85}) }
		Let $A,B \in \Mn$. If $|A|\le B$, then $\rho(A) \le \rho(|A|) \le \rho(B)$. 
	\end{lem}

	\begin{lem}\label{thm_inverseMmatrix}\emph{(\cite{McDonald95}) }
		Let $P \in \Mn$ be the inverse of a nonsingular M-matrix. Then $P \ge 0$ and $P_{jk} \ge P_{ji} P_{ii}^{-1} P_{ik}, \forall i,j,k =1, \ldots, N.$
	\end{lem}
	
	\begin{lem}\label{thmWoodbury}\emph{(\cite[p.~258]{Higham02})}
		Let $A \!\in\! \mathbb{R}^{n\times n}, B \!\in\! \mathbb{R}^{n\times r}, C \!\in\! \mathbb{R}^{r\times r}$, $D \!\in\! \mathbb{R}^{r\times n}.$ The following holds when the indicated inverses exist
		\begin{equation}  
		(A \!-\! BC^{-1}D)^{-1} \!=\! A^{-1} \!+\! A^{-1}B(C \!-\! DA^{-1}B)^{-1}DA^{-1}\label{eqWoodbury}
		\end{equation} 
	\end{lem}

	\subsection{Proof of Lemma \ref{lemMatrixSupermodularity}} \label{proofMatrixSupermodularity}
	First, we show that  $(L_{\beta} + \Gamma_\mathcal{S})^{-1}$ is nonincreasing in $\mathcal{S}$. Let $D_{\mathcal{S}} = \diag(W\1 + \boldsymbol\beta + \boldsymbol\alpha_{\mathcal{S}})$ and note that $\rho \big( D_{\mathcal{S}}^{-1}W \big) < 1$ (cf. Lemma \ref{lemStability}). 
	Thus, $ (I- D_{\mathcal{S}}^{-1}W)^{-1}= \sum_{i=0}^{\infty} (D_{\mathcal{S}}^{-1}W)^i,$ 
	and 
	\begin{align}
	(L_{\beta} + \Gamma_\mathcal{S})^{-1} 
	= (D_{\mathcal{S}} -W)^{-1} 
	= \textstyle\sum_{i\ge 0} (D_{\mathcal{S}}^{-1}W)^i D_{\mathcal{S}}^{-1} \label{eqMatrixMonotone}
	\end{align}
	which is clearly nonnegative. Moreover, for any $\mathcal{S} \subseteq \mathcal{T} \subseteq \mathcal{V}$, we have $\0 \le D_{\mathcal{T}}^{-1} \le D_{\mathcal{S}}^{-1}$, which together with \eqref{eqMatrixMonotone} implies that 
	$(L_{\beta} + \Gamma_\mathcal{T})^{-1} \le (L_{\beta} + \Gamma_\mathcal{S})^{-1}.$ 
	This proves the second inequality in \eqref{eqMatrixSupermodular}. 
	Next, we prove the first inequality in \eqref{eqMatrixSupermodular}.  
	By letting $P\! = (L_\beta\!+\! \Gamma_\mathcal{S})^{-1}$, $Q\!= (L_\beta \!+\! \Gamma_{\mathcal{S}\cup \{k\}})^{-1}$ and using  \eqref{eqWoodbury}, this is equivalent to proving that
	$	 P_{(v)}P^{(v)}(\alpha_v^{-1} + P_{vv})^{-1} 	\geq
	Q_{(v)}Q^{(v)}(\alpha_v^{-1} + Q_{vv})^{-1}.
	$
	We will show that this holds element-wise, i.e., for $\forall i,j\in \mathcal{V}$
	\begin{align}
	{P_{iv}P_{vj}}{(\alpha_v^{-1} + P_{vv})^{-1}}  \geq
	{Q_{iv}Q_{vj}}{(\alpha_v^{-1} + Q_{vv})^{-1}}.
	\label{eq_JSupermodular2}
	\end{align}
	By \eqref{eqWoodbury}, we have 
	$	Q_{ij} = P_{ij} - P_{ik}P_{kj}/(\alpha_k^{-1} + P_{kk}), \forall i,j\in \mathcal{V}.$
	Thus, \eqref{eq_JSupermodular2} is equivalent to 
	$\frac{\alpha_v^{-1} + Q_{vv}}{\alpha_v^{-1} + P_{vv}}P_{iv}P_{vj}
	\geq
	(	P_{iv} \!-\! \frac{P_{ik}P_{kv}}{\alpha_k^{-1} + P_{kk}} ) 	
	(	P_{vj} \!-\! \frac{P_{vk}P_{kj}}{\alpha_k^{-1} + P_{kk}}  )$
	or, by rearranging terms, this is
	\begin{align}
	\textstyle \frac{P_{vk}P_{kv}P_{iv}P_{vj}}{(\alpha_v^{-1} + P_{vv})}  
	{+} \frac{P_{ik}P_{kv}P_{vk}P_{kj}}{(\alpha_k^{-1} + P_{kk})} 
	\le
	P_{ik}P_{kv}P_{vj}
	+ P_{iv}P_{vk}P_{kj}.\nnb	
	\end{align}
	To show this, 
	first note that $P$ is the inverse of a nonsingular M-matrix. Thus, by Lemma \ref{thm_inverseMmatrix} and the fact that $\alpha_v^{-1}\ge 0$, we have $P_{ik} \ge P_{iv}P_{vk}/P_{vv} \ge P_{iv}P_{vk}/(\alpha_v^{-1} + P_{vv})$. 
	Next, multiplying both sides of this relation with $P_{kv}P_{kj}\ge 0$ yields 
	$P_{vk}P_{kv}P_{iv}P_{vj}/(\alpha_v^{-1} + P_{vv}) \le P_{ik}P_{kv}P_{vj}.$ 
	Similarly, we have 
	$P_{ik}P_{kv}P_{vk}P_{kj}/(\alpha_k^{-1} + P_{kk}) \le P_{iv}P_{vk}P_{kj}.$ 
	%
	Adding the last two relations yields the desired result. \hfill \qed

	\vspace{-2mm}
	\subsection{Proof of Lemma \ref{lemComposition}} \label{proofComposition}
	Let $\phi = f\circ F$. We will show that $\phi(\mathcal{S}) + \phi(\mathcal{T}) \le \phi(\mathcal{S}\cup \mathcal{T}) + \phi(\mathcal{S} \cap \mathcal{T})$ for any $\mathcal{S} \subseteq \mathcal{T} \subseteq \mathcal{V}$. 
	First, since $F$ is decreasing, we have 
	$F(\mathcal{S}\!\cup\! \mathcal{T}) \!\le\! \min\{F(\mathcal{S}),F(\mathcal{T})\} \!\le\! \max\{F(\mathcal{S}),F(\mathcal{T})\} \!\le\! F(\mathcal{S}\!\cap\! \mathcal{T}).$ 
	As a result, $\phi(\mathcal{S}\cup \mathcal{T}) = f\big( F(\mathcal{S}\cup \mathcal{T}) \big) \le  f\big( F(\mathcal{S} \cap \mathcal{T}) \big) = \phi(\mathcal{S} \cap \mathcal{T})$ since $f$ is increasing. This proves that $\phi$ is nonincreasing. 
	
	Next, we have that  there exist $a_1, a_2 \in [0,1]$ such that 
	\begin{align}
	F(\mathcal{S}) &= a_1 F(\mathcal{S}\cup \mathcal{T}) + (1-a_1) F(\mathcal{S}\cap \mathcal{T}) \label{FS}\\
	F(\mathcal{T}) &= a_1 F(\mathcal{S}\cup \mathcal{T}) + (1-a_2) F(\mathcal{S}\cap \mathcal{T}). \label{FT}
	\end{align}
	Adding these equations gives $
	F(\mathcal{S}) + F(\mathcal{T}) 
	= (a_1 + a_2) F(\mathcal{S}\cup \mathcal{T}) + (2-a_1-a_2) F(\mathcal{S}\cap \mathcal{T})$, 
	whose left side is less than $F(\mathcal{S}\!\cup\! \mathcal{T}) \!+\! F(\mathcal{S}\!\cap\! \mathcal{T})$ by supermodularity of $F$. Thus
	$(a_1 + a_2) F(\mathcal{S}\cup \mathcal{T}) + (2-a_1-a_2) F(\mathcal{S}\cap \mathcal{T}) \le F(\mathcal{S}\cup \mathcal{T}) + F(\mathcal{S}\cap \mathcal{T}).$ Rearranging terms, we have 
	$(1-a_1 -a_2)(F(\mathcal{S}\cap \mathcal{T}) - F(\mathcal{S}\cup \mathcal{T})) \le \0$
	which together with monotonicity of $F$ yields  
	$ a_1 + a_2 \ge 1.$ 
	Now by convexity of $f$ and \eqref{FS} we have $\phi(\mathcal{S}) {=} f\big(F(\mathcal{S}) \big) {=} f\big( a_1 F(\mathcal{S}\cup \mathcal{T}) + (1-a_1) F(\mathcal{S}\cap \mathcal{T})  \big) 
	\le a_1 f\big(F(\mathcal{S}\cup \mathcal{T})\big) + (1-a_1) f\big( F(\mathcal{S}\cap \mathcal{T})  \big)$. 
	So, $\phi(\mathcal{S}) \le a_1 \phi(\mathcal{S}\cup \mathcal{T}) + (1-a_1)\phi(\mathcal{S} \cap \mathcal{T})$. 
	Similarly, convexity of $f$ and \eqref{FT} imply $\phi(\mathcal{T}) \le a_2 \phi(\mathcal{S}\cup \mathcal{T}) + (1-a_2)\phi(\mathcal{S} \cap \mathcal{T}).$ Thus, 
	\begin{align}
	&\phi(\mathcal{S}) + \phi(\mathcal{T}) 
	\le (a_1 \!+\! a_2) \phi(\mathcal{S}\cup \mathcal{T}) \!+\! (2\!-\!a_1\!-\!a_2)\phi(\mathcal{S} \cap \mathcal{T}) \nnb\\
	&=  \phi(\mathcal{S}\cup \mathcal{T}) + \phi(\mathcal{S} \cap \mathcal{T}) +(a_1+a_2-1)\big[ \phi(\mathcal{S}\cup \mathcal{T}) - \phi(\mathcal{S} \cap \mathcal{T}) \big] \nnb
	\end{align}
	Since  $a_1+a_2 \ge 1$ and $\phi$ is nonincreasing, we conclude that $\phi(\mathcal{S}) + \phi(\mathcal{T}) \le \phi(\mathcal{S}\cup \mathcal{T}) + \phi(\mathcal{S} \cap \mathcal{T})$ as desired.\hfill \qed%
	
	\vspace{-2mm}
	\section*{Acknowledgment}	
    The authors are grateful to Prof. Terence Tao for providing the approach used in the proof of Lemma~\ref{lemPVPVP}. This work was supported in part 
	by the Air Force Office of Scientific Research through MURI AFOSR Grant {\#FA9550-09-1-0538}.%
	
	\vspace{-2mm}
	\bibliographystyle{IEEEtran}
	\bibliography{IEEEabrv,RefInfluenceAbrv,RefOptim1,RefProposal}

\end{document}